\documentclass{hjm1}
\usepackage{amsmath}
\usepackage{amssymb}
\usepackage{amsthm}
\usepackage{amsfonts}
\usepackage{mathrsfs}
\usepackage{commath}
\usepackage{tikz}
\usepackage{cancel}
\usepackage{graphicx}
\graphicspath{ {./images/} }
\usepackage{enumerate}

\usepackage{ifthen}

\usepackage{listings}

\usepackage{hyperref,url,pst-node}
\hypersetup{citecolor=red, linkcolor=blue, colorlinks=true}

\usepackage{mathtools}

\newcommand{\Conv}{
  \mathop{\scalebox{3.0}{\raisebox{-0.1ex}{$\ast$}}
  }
}

\newcommand{\CANTOR}[3]{
\ifnum 0 < \numexpr#1 {
    \CANTOR{\numexpr#1-1}{#2*0.333333}{#3}
    \CANTOR{\numexpr#1-1}{#2*0.333333}{#3+0.666666*#2}}
\else{
    \draw[black,line width=0.5cm] (#3,0)--(#3+#2,0);}
\fi{}
}

\newcommand{\CANTORM}[3]{
\ifnum 0 < \numexpr#1 {
    \CANTOR{\numexpr#1-1}{#2*0.4}{#3}
    \CANTOR{\numexpr#1-1}{#2*0.4}{#3+0.6*#2}}
\else{
    \draw[black,line width=0.5cm] (#3,0)--(#3+#2,0);}
\fi{}
}

\newenvironment{myquote}[1]
  {\list{}{\leftmargin=#1\rightmargin=#1}\item[]}
  {\endlist}

\theoremstyle{plain}
\newtheorem{example}{Example}[section]


\address{School of Mathematical and Physical Sciences, 
   University of Newcastle, \newline
\hspace*{\parindent}University Drive, Callaghan NSW 2308, Australia}
\email{james.evans10@uon.edu.au}

\title{The Ghost Measures of Affine Regular Sequences}
\author{James Evans}
\date{}

\begin{document}

\maketitle

\begin{abstract}
    The $k$-regular sequences exhibit a self-similar behaviour between powers of $k$. One way to study this self-similarity is to attempt to describe the limiting shape of the sequence using measures, which results in an object which we call the \textit{ghost measure}. The aim of this paper is to explicitly calculate this measure and some of its properties, including its Lebesgue decomposition, for the general family of \textit{affine} $2$-regular sequences.
\end{abstract}

\section{Introduction}

The concept of self-similarity is prevalent in many areas of mathematics, such as symbolic dynamics and fractals. The examples that concern us presently are the $k$-regular sequences, a generalisation of automatic sequences which have relevance to number theory and Mahler functions. The definition of a $k$-regular sequence involves the base $k$-representation of $n$, giving the sequences a recurrent behaviour between powers of $k$.

A common theme in the study of self similarity is the construction of a full self-similar picture via a limit process using finite approximations; constructing an infinite fixed point of a substitution by iteration, constructing the cantor set by intersecting level sets, etc. When studying a $k$-regular sequence $f$, we can use the following natural idea. We take the terms of the sequence between adjacent powers of $k$, interpret them as weights of a Dirac comb in $[0,1)$ and then normalise to create a sequence of probability measures on the $1$-torus $\mathbb{T}$,
$$
\mu_{N}=\frac{1}{\Sigma(N)}\sum_{n=0}^{k^{N+1}-k^{N}-1}f(k^{N}+n)\delta_{\frac{n}{k^{N}(k-1)}}.
$$
If this sequence vaguely converges, we call the limit $\mu$, the \textit{ghost measure} of the sequence. This procedure was first demonstrated by Baake and Coons \cite{SternSequenceMeasure}, who proved that the ghost measure of Stern's Diatomic sequence is singular continuous. Subsequently Coons, Evans and Manibo \cite{CEMRegularPaper} have explored the general question existence of ghost measures. This paper studies specific examples of the measure arising from the general family of \textit{affine} $2$-regular sequences. These are sequences obeying the relations
\begin{align}
    f(2n)=A_{0}f(n)+b_{0},\quad f(2n+1)=A_{1}f(n)+b_{1}, \label{eqn1}
\end{align}
where the coefficients are non-negative integers, not all zero. We first prove the following classification of the Lebesgue decomposition of their ghost measures.
\begin{thm} \label{THEOREM}
Suppose that $f(n)$ is an affine $2$-regular sequence as in \eqref{eqn1}. Then its ghost measure $\mu$ exists and its Lebesgue decomposition is as follows.
\begin{enumerate}
    \item Suppose $b_{0}=b_{1}=0$:
    \begin{itemize}
        \item[(A)] if $A_{0}=A_{1}\neq 0$ then $\mu$ equals Lebesgue measure $\lambda$,
        \item[(B)] if $A_{0}\neq A_{1}$, neither equal $0$, then $\mu$ is singular continuous,
        \item[(C)] if $A_{0}$ or $A_{1}$ is $0$, then $\mu$ is pure point, and equals $\delta_{0}$.
    \end{itemize}
    \item Suppose $b_{0}+b_{1}\neq 0$:
    \begin{itemize}
        \item[(A)] if $A_{0}+A_{1}\leqslant 2$, then $\mu=\lambda$,
        \item[(B)] if $A_{0}=A_{1}>1$, then $\mu$ is absolutely continuous,
        \item[(C)] if $A_{0}\neq A_{1}$, neither $0$, then $\mu$ is singular continuous,
        \item[(D)] if $A_{0}$ or $A_{1}=0$, the other greater than $2$, then $\mu$ is pure point.
    \end{itemize}
\end{enumerate}
\end{thm}

We then give more detail about these measures, including what the Radon--Nikodym derivative is in \textbf{2B}, where the measure is concentrated in \textbf{2C} and what the point weights are in \textbf{2D} (bold labels refer to the cases of Theorem \ref{THEOREM}). Here we see self-similarity cropping up again; not only is it describing the limit shape that the sequence, $\mu$ itself also has many fractal-like characteristics.

Before moving on, we explain the name \textit{ghost measure}. Neither \cite{SternSequenceMeasure} nor \cite{CEMRegularPaper} give a name to their construction. The inspiration for ours comes from Berkeley's critique of infinitesimals in \textit{The Analyst} \cite{TheAnalyst}, when he says that they are

\begin{myquote}{0.9cm}
   $\ldots$ neither finite quantities nor quantities infinitely small, nor yet nothing. May we not call them \textbf{the ghosts of departed quantities?}
\end{myquote}
The values $f(n)$ are (usually) much smaller than the sum of all terms, so the individual pure points of the $\mu_{N}$ disappear in the averaging as $N$ tends to infinity. The measure $\mu$ is the ethereal imprint that is left behind, the ghost of the departed pure points of the $\mu_{N}$. There is something \textit{spooky} about this. The sequence is discrete; but as we shall see, $\mu$ is often continuous, making the measure a strange, uncountable reflection of a countable structure.

\section{Preliminaries}

A sequence $\{f(n)\}_{n=0}^{\infty}$ is $k$-regular if there are integer $d\times d$ matrices $C_{0},\ldots,C_{k-1}$ and $d\times 1$ vectors $L,M$ (together called the linear representation of the sequence) such that $f(n)=L^{T}\cdot C_{i_{l}}C_{i_{l-1}}\cdots C_{i_{0}}\cdot M$, where $i_{l}\cdots i_{1}i_{0}$ is the base $k$ representation of $n$. For more details see \cite{AutomaticSequences,RingRegularSequences,RingRegularSequences2}.

\begin{dfn}[Ghost Measure]\label{GhostMeasure}
Let $\{f(n)\}_{n=0}^{\infty}$ be a non-negative $k$-regular sequence. The intervals $\mathcal{R}_{N}=[k^{N},k^{N}+1,\ldots,k^{N+1}-1)$, $N\geqslant 0$ are called the fundamental regions. Define $\Sigma(N)=\sum_{n \in \mathcal{R}_{N}}f(n)$. Identify the $1$-torus $\mathbb{T}$ as $[0,1)$ with addition modulo $1$ and let $\delta_{x}$ denote the Dirac delta measure at $x$. The \textit{$N$th approximant to the ghost measure} is the pure point probability measure
$$
\mu_{N}=\frac{1}{\Sigma(N)}\sum_{n=0}^{k^{N}(k-1)-1}f(k^{N}+n)\delta_{\frac{n}{k^{N}(k-1)}}.
$$
The ghost measure is defined to be the vague limit\footnote{A sequence of Borel measures $\mu_{n}$ on $\mathbb{T}$ vaguely converges to another Borel measure $\mu$ if for every continuous function $f$, $\int f d \mu_{n} \to \int f d\mu$ as $n \to \infty$. For more details, see  \cite{baake_grimm_2017,RudinFourierAnalysis,RudinAnalysis}.} $\mu=\lim_{N \to \infty}\mu_{N}$ (if it exists).
\end{dfn}

There is some choice in this definition. It might sometimes be more natural to use $[0,k^{N}-1]$ rather than $[k^{N},k^{N+1}-1]$ as the fundamental regions. One can also go beyond $k$-regular sequences and apply the idea to any sequence $f(n)$ and series of intervals $\mathcal{R}_{N}$ whose length tends to infinity, but only with sufficient regularity will this result in some interesting objects.

Our focus shall be on the Lebesgue decomposition of ghost measures, so we review the concept briefly. In the following $\lambda$  denotes Lebesgue measure on $\mathbb{T}$.

\begin{thm}[Lebesgue decomposition theorem] \cite[Chapter 6]{RudinAnalysis}
Suppose that $\mu$ is a finite Borel measure on $\mathbb{T}$. Then there exists three other measures $\mu_{pp}$, $\mu_{sc}$ and $\mu_{ac}$ such that $\mu=\mu_{pp}+\mu_{sc}+\mu_{ac}$, where
\begin{itemize}
    \item $\mu_{pp}$ is pure point: it is a countable sum of weighted Dirac delta measures,
    \item $\mu_{sc}$ is singular continuous: it is concentrated on an uncountable lebesgue measure zero set,
    \item $\mu_{ac}$ is absolutely continuous: there is a function (its Radon--Nikodym derivative) $g\in L^{1}(\mathbb{T})$ such that $d\mu_{ac}=g d\lambda$.
\end{itemize}
\end{thm}

The theorem is useful because each component has very different properties, so knowing the decomposition provides a qualitative assessment of how the measure and the object from which it was constructed are behaving. This is particularly true when the measure has pure type, that is, only one component is non-zero.

\section{The Affine regular sequences} \label{EXAMPLES}

As a reminder, the affine $2$-regular sequences are sequences obeying $f(2n)=A_{0}f(n)+b_{0}$, $f(2n+1)=A_{1}f(n)+b_{1}$. The relations may be inconsistent at $n=0$, so we simply start with $f(1)$. We restrict the coefficients to be non-negative integers, not all zero, to guarantee that the sequences are non-negative and that $\mu$ will be a probability measure. Despite their simplicity, the affine $2$-regular sequences include many famous and important examples \cite{RingRegularSequences,RingRegularSequences2}, which we briefly review. Where relevant, the plot on the left shows the sequence from $0$ up to $2^{N}$ while the right shows $F_{N}(x)=\mu_{N}([0,x])$, for some large $N$. These plots clearly show the self-similar behaviour that we are trying to describe

\begin{example} The constant sequence $f(n)=1$, obeying $f(2n)=1$ and $f(2n+1)=1$ is trivially an affine $2$-regular sequence. The ghost measure of this sequence is Lebesgue measure, and hence absolutely continuous.
\end{example}

\begin{example}
The sequence $f(n)=n$ obeys $f(0)=0$, $f(2n)=2f(n)$, $f(2n+1)=2f(n)+1$. Its ghost measure is absolutely continuous and equals $\frac{2+2x}{3} \lambda$.
\end{example}

\begin{example}[Gould's Sequences]
These are among the oldest $2$-regular sequences to be described as such. The first is $g(n)=$ the number of $1$'s in $n$'s binary expansion, which obeys $g(2n)=g(n)$, $g(2n+1)=g(n)+1$, $g(0)=0$. It has $\Sigma(N)=2^{N-1}(N+2)$ and ghost measure  $\mu=\lambda$. The second version is $G(n)=2^{g(n)}$, obeying $G(2n)=G(n)$, $G(2n+1)=2G(n)$, $G(0)=1$. It has $\Sigma(N)=2\cdot3^{N}$ and a singular continuous ghost measure. It turns out that $G(n)$ counts the odd numbers in the $n$th row of Pascal's triangle \cite{BinomialCoefficients}.
\begin{figure}[h]
    \centering
    \includegraphics[scale=0.4]{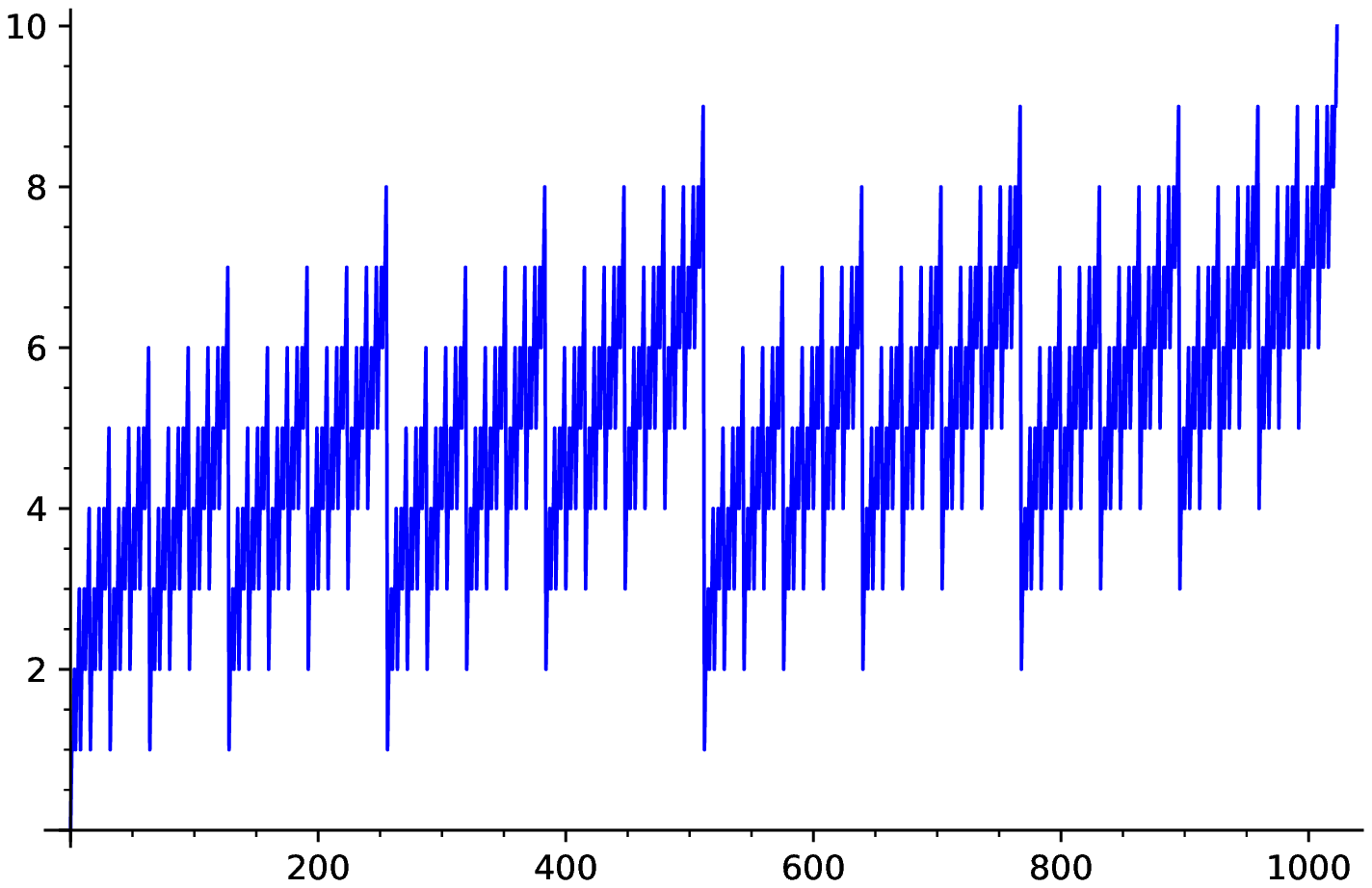}
    \includegraphics[scale=0.4]{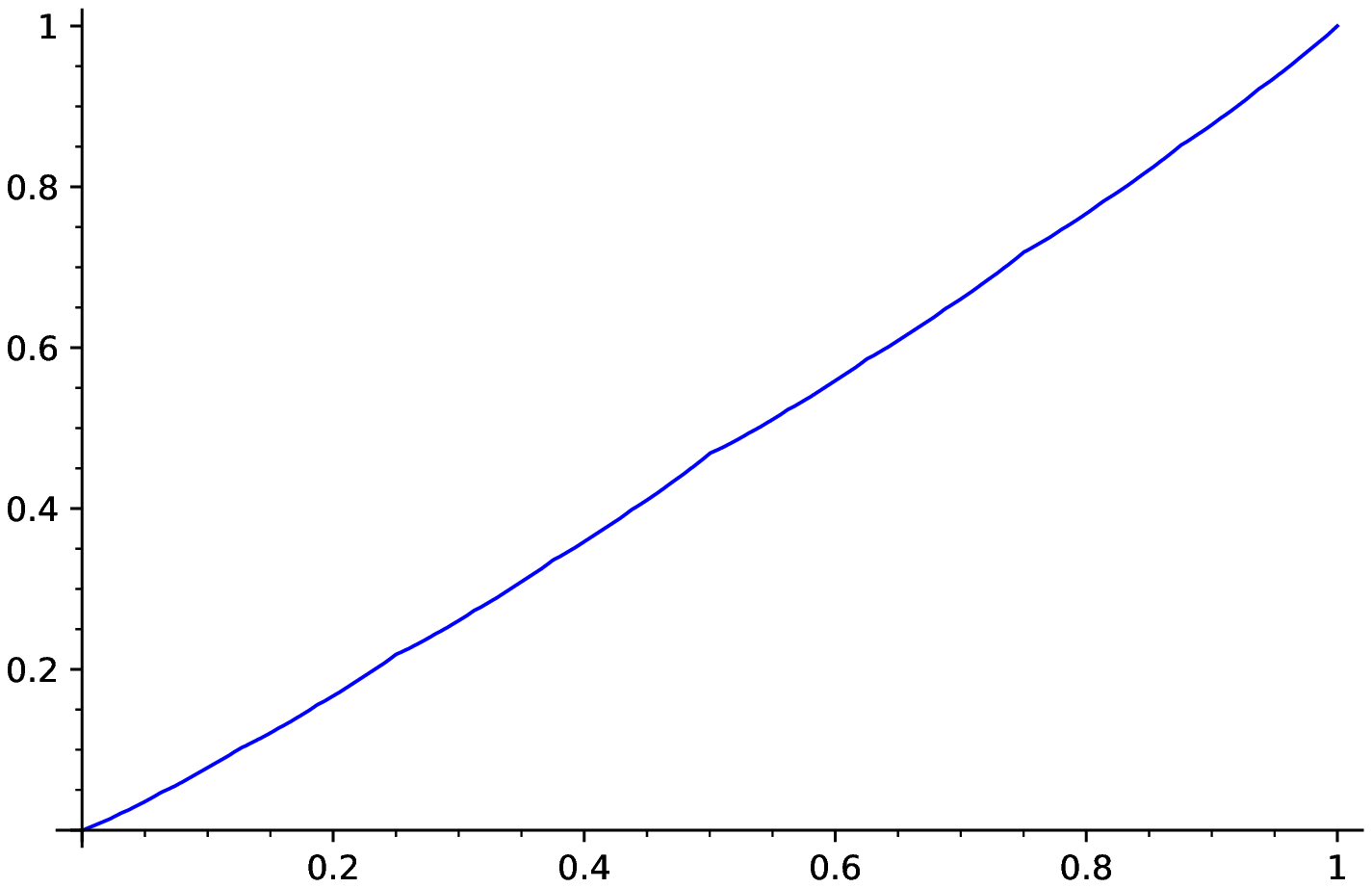}
    \caption{$g(n)$}
\end{figure}

\begin{figure}[h]
    \centering
    \includegraphics[scale=0.4]{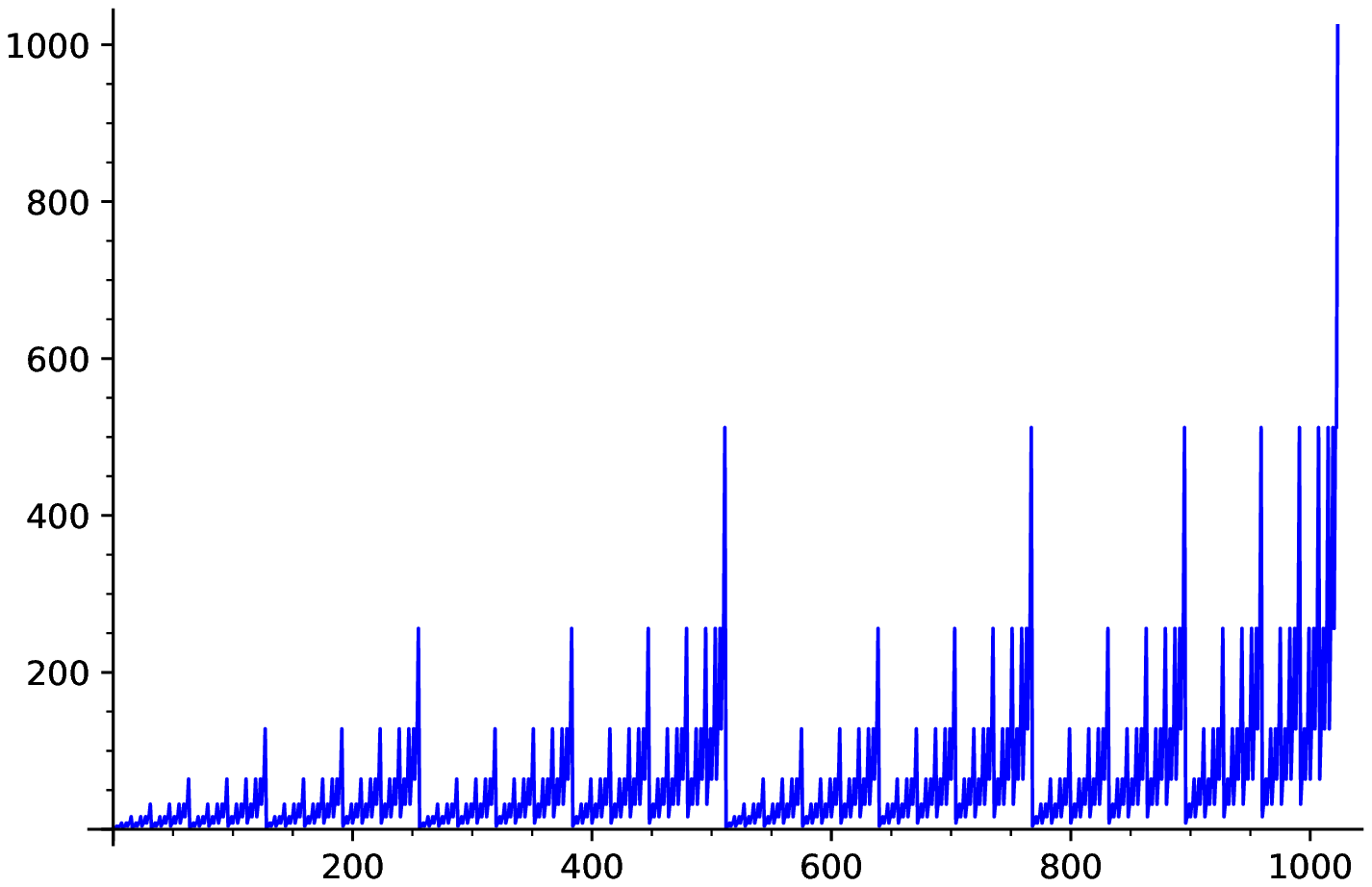}
    \includegraphics[scale=0.4]{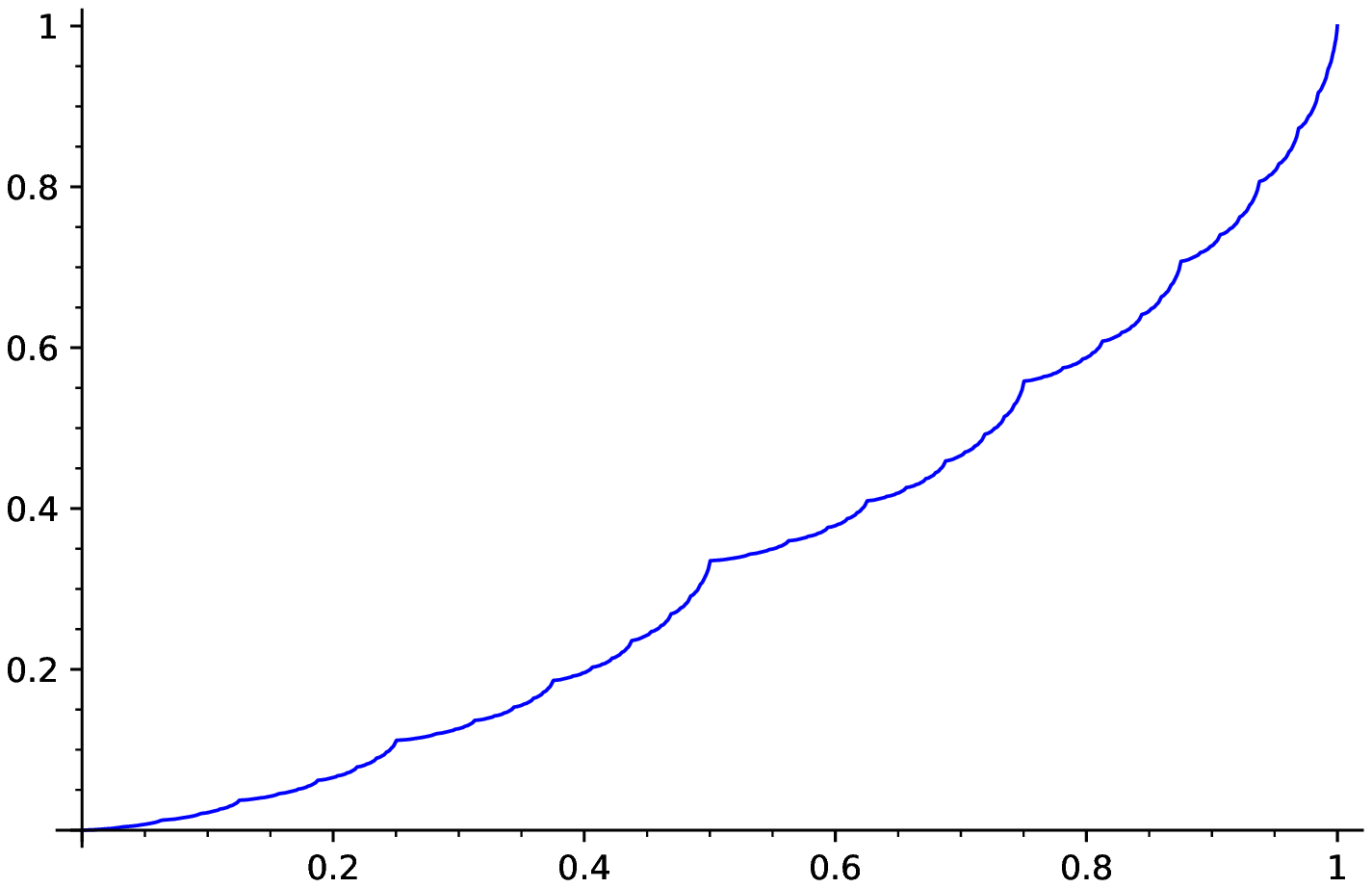}
    \caption{$G(n)$}
\end{figure}

It might be surprising that the ghost measure of $g(n)$ is Lebesgue measure, given the sequence's irregularity. However, on large scales most of the mass is evenly distributed, as shown by $F_{N}(x)$ being essentially a straight line. For $G(n)$ the peaks are more pronounced resulting in a singular measure. Interestingly, $G(n)$'s $F_{N}(x)$ is among the first explicit examples of strictly increasing, singular continuous functions that were constructed by Salem \cite{SalemStrictlyIncreasingSingularFunctions}. One can use this fact as part of another proof of \textbf{2C}.
\end{example}

\begin{example}[The Ruler sequences]
As with Gould's, there are two versions. There is $r(n)$ the $2$-adic valuation of $n$ and $R(n)=2^{r(n)}$. The relations for $r(n)$ are $r(2n)=r(n)+1$, $r(2n+1)=0$, while $R(n)$ obeys $R(2n)=2R(n)$, $R(2n+1)=1$. These sequences have $\Sigma(N)=2^{N}-1$ and $2^{N-1}(N+2)$ respectively. Both have ghost measure equal to Lebesgue measure. Even though both sequences have large peaks, in the limit all of $\mu_{N}$'s mass is concentrated in the small, uniformly distributed pure points. The convergence can be quite slow, especially for $R(n)$ as shown by large discontinuities in $F_{N}(x)$, even for the large $N$ chosen.

\begin{figure}[h]
    \centering
    \includegraphics[scale=0.4]{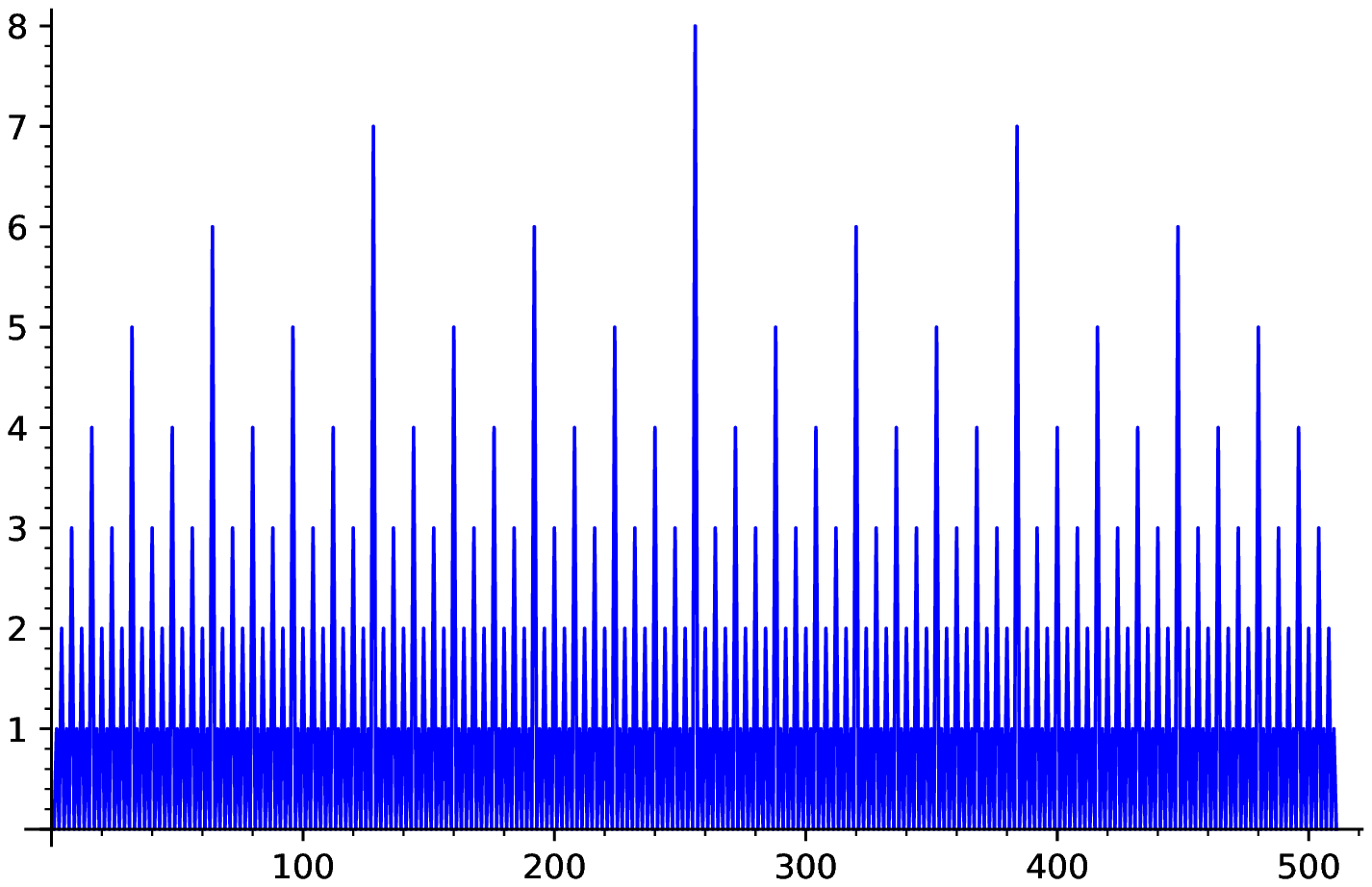}
    \includegraphics[scale=0.4]{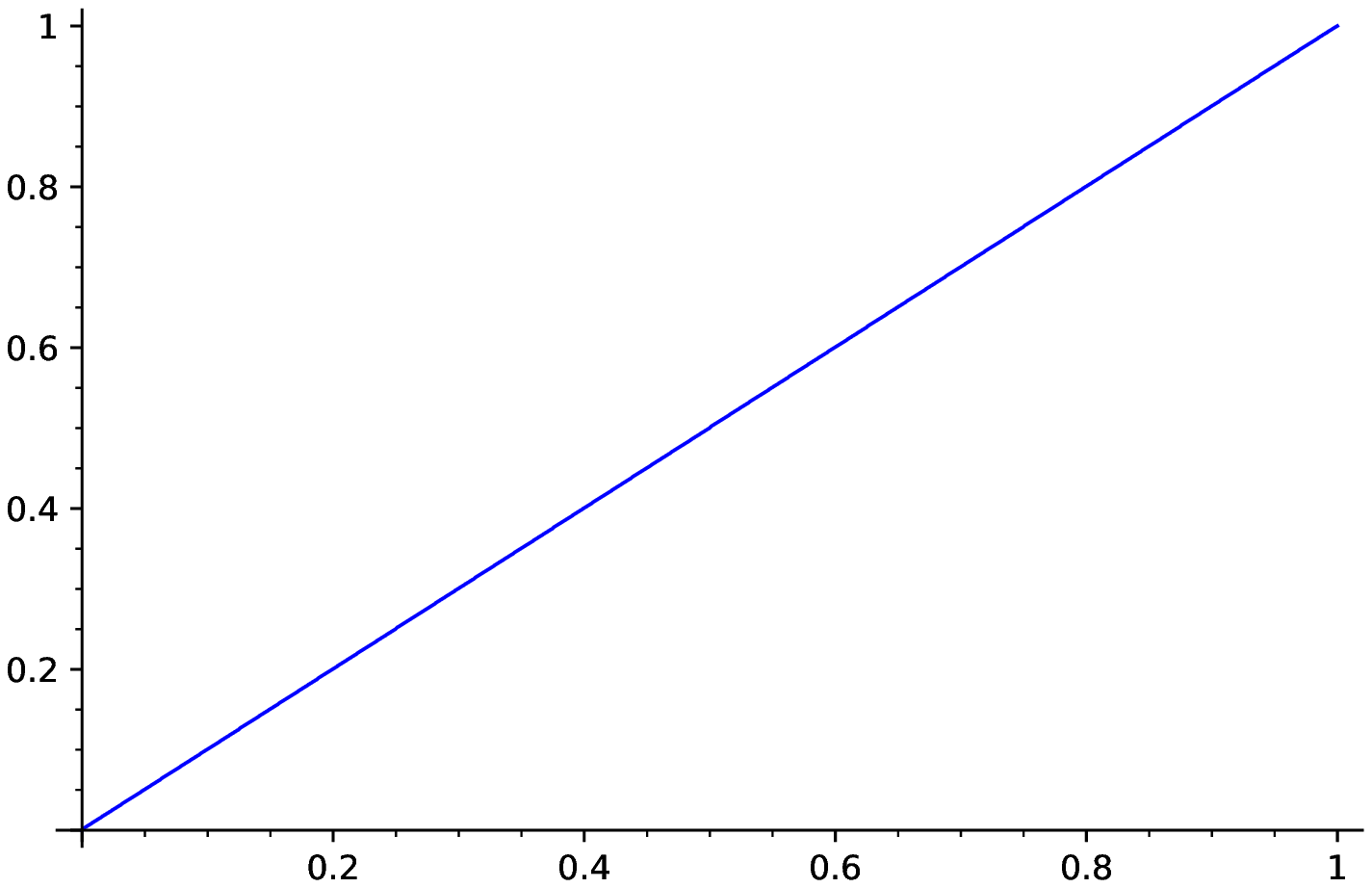}
    \caption{$r(n)$}
\end{figure}

\begin{figure}[h]
    \centering
    \includegraphics[scale=0.4]{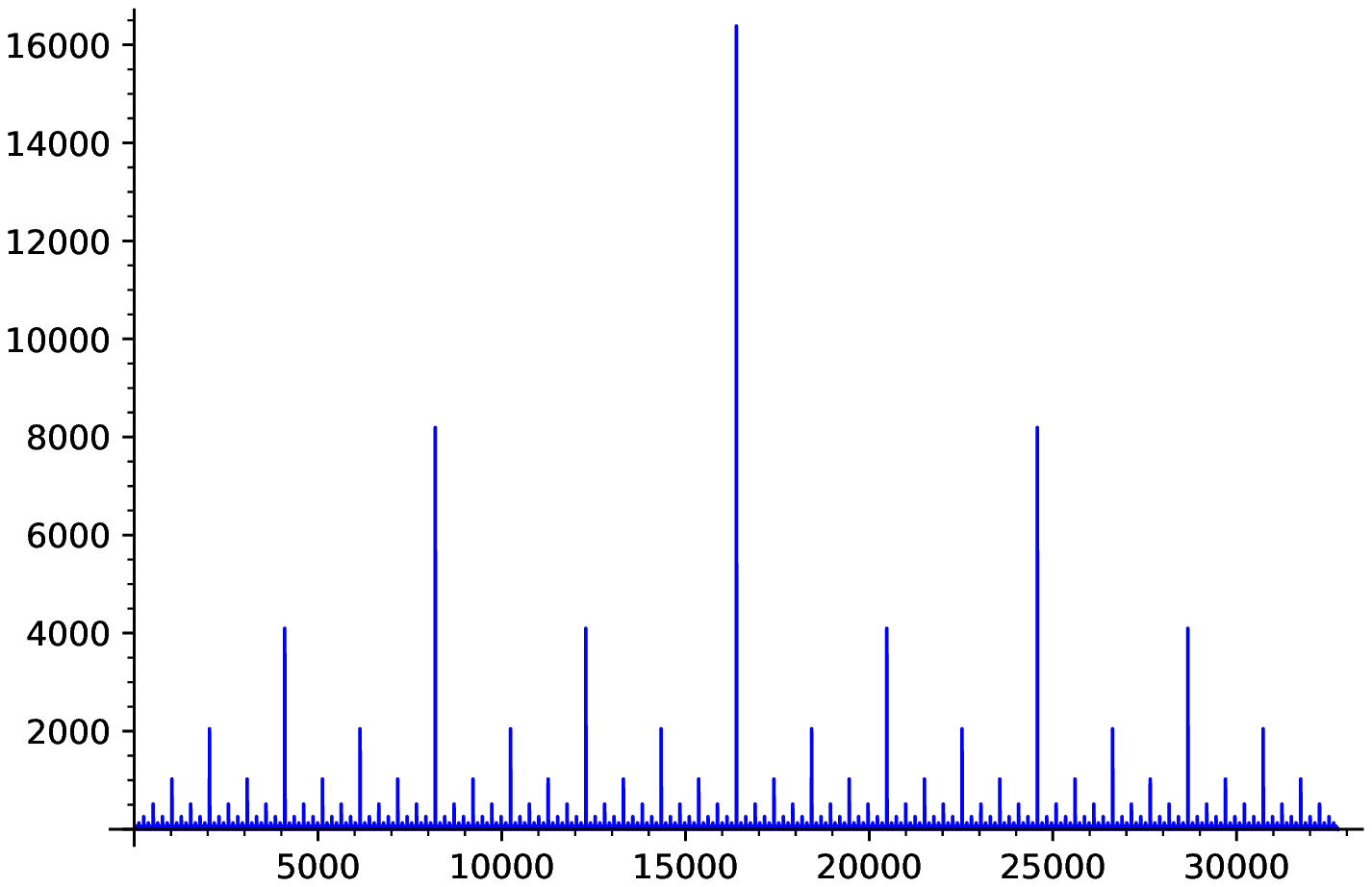}
    \includegraphics[scale=0.4]{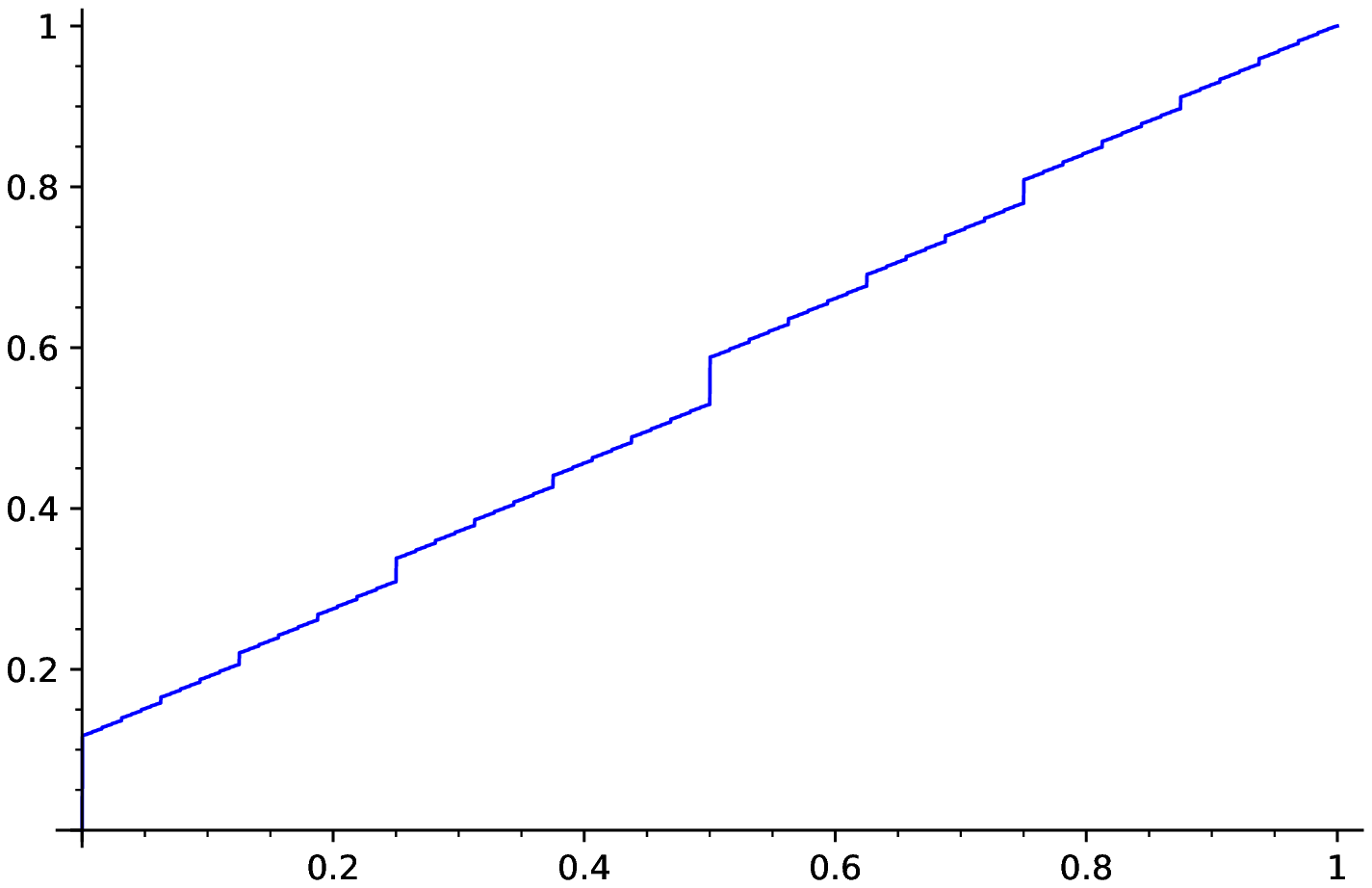}
    \caption{$R(n)$}
\end{figure}
\end{example}

\begin{example}[Missing digit examples] 
Consider the sequences $z(2n)=d\cdot z(n)$, $z(2n+1)=d\cdot z(n)+j$, $z(0)=0$, $d\geqslant 2$, $1 \leqslant j \leqslant d-1$. These enumerate the numbers with only $0$ and $j$ in their base $d$ expansion. Named examples include

\begin{itemize}
    \item The Cantor sequence: the numbers with no $1$ in their ternary expansion. It obeys $c(2n)=3c(n)$, $c(2n+1)=3c(n)+2$, $c(0)=0$.
    \item No arithmetic progressions: define $e(0)=0$, $e(1)=1$. Then $e(n)$, $n>1$, is the least integer not in arithmetic progression with any two previous terms. 
    The sequence obeys $e(2n)=3e(n)$, $e(2n+1)=3e(n)+1$, $e(0)=0$.
    \item The Moser--De Bruijn sequence: enumerates the numbers that are sums of distinct powers of $4$. Obeys $m(2n)=4m(n)$, $m(2n+1)=4m(n)+1$, $m(0)=0$.
\end{itemize}

\begin{figure}[h]
    \centering
    \includegraphics[scale=0.4]{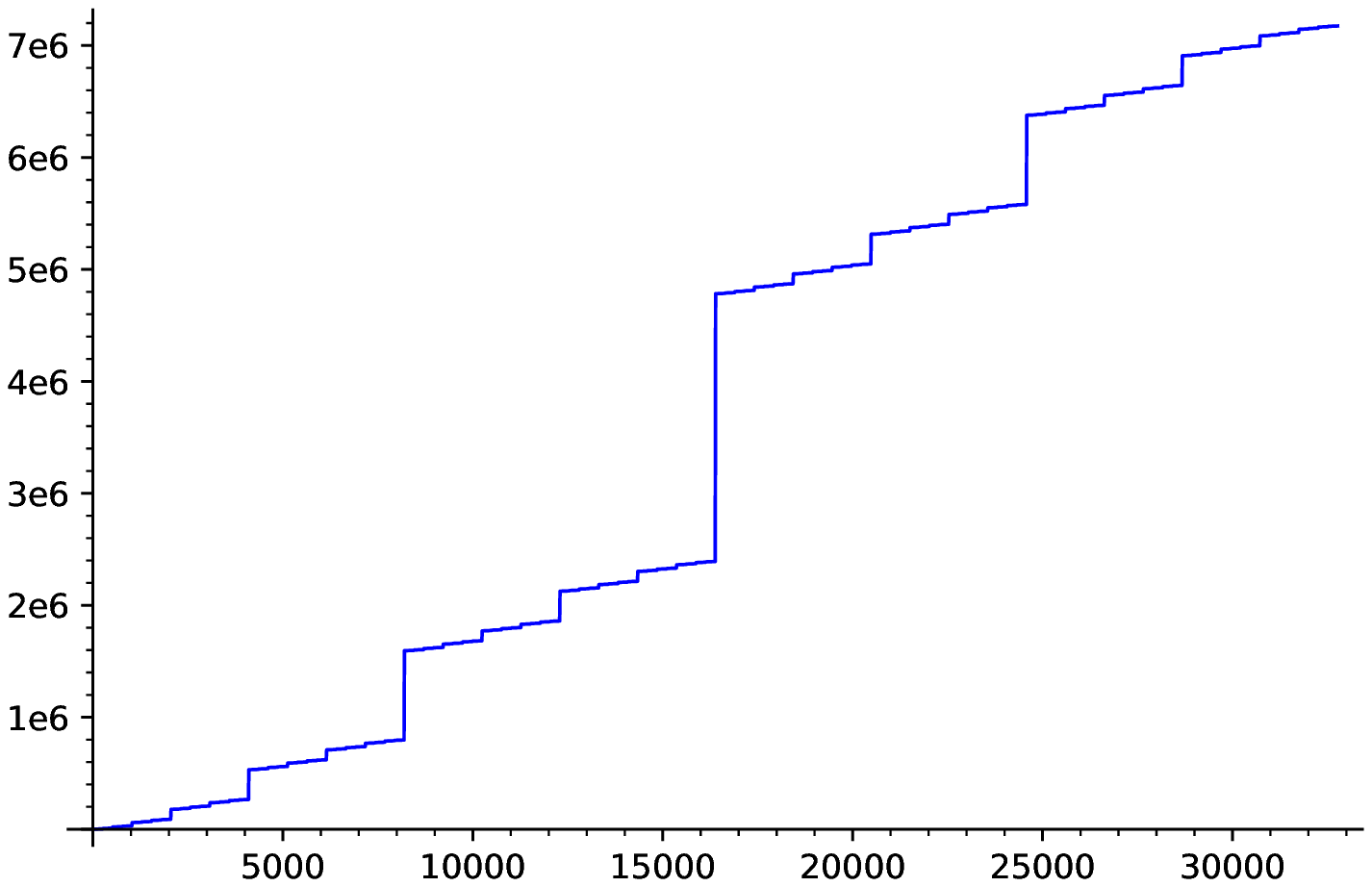}
    \includegraphics[scale=0.4]{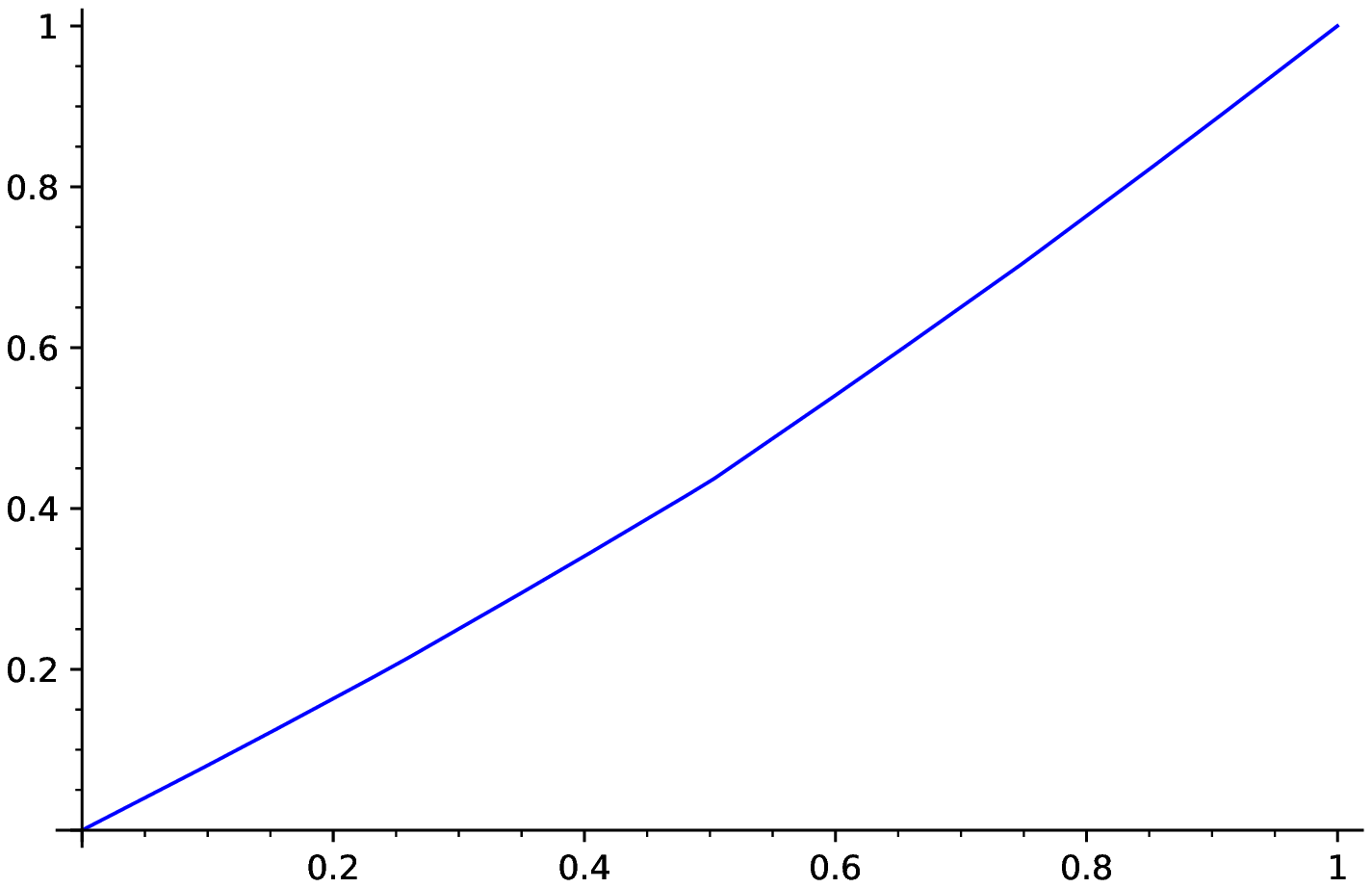}
\end{figure}

Their ghost measures are all absolutely continuous. This is clear because the sequences are strictly increasing, so $\mu_{N}$'s mass cannot be concentrating towards any measure zero set, ruling out a singular part.

\end{example}

\begin{example}[A trivial pure-point example]
Suppose that $N(n)=1$ if and only if $n=2^{k}$ for some $k$, and zero elsewise (i.e. $N(2n)= N(n)$, $N(2n+1)=0$, $N(1)=1$). Then for all $N$, $\mu_{N}$ (and so also $\mu$) equals $\delta_{0}$. This is probably the simplest example with a pure-point ghost measure, and it exemplifies the degenerate behaviour required for $\mu$ to have a pure point part.
\end{example}

The following sequences technically do not fall under Theorem \ref{THEOREM} because they have negative coefficients in their recurrence relations. We assume positivity so as to avoid complications in the proofs. But, with some more care, the same working applies to many more sequence than Theorem \ref{THEOREM}. That our work can deal with this very wide range of important examples emphasises its usefulness.

\begin{example}[The Josephus Sequence]
Probably the regular sequence with the most interesting (and violent) story associated to it \cite{Numberphile}. It obeys $J(2n)=2J(n)-1$, $J(2n+1)=2J(n)+1$, $J(1)=1$. Between powers of $2$ the sequence is linearly increasing ($J(2^n+k)=2k+1$, for  $0 \leqslant k <2^{n}$), so $\Sigma(N)=4^{n}$ and the ghost measure is absolutely continuous and equals $2x \cdot \lambda$.
\end{example}

\begin{example}[The Thue--Morse Sequence]
The Thue--Morse sequence is probably the most important and ubiquitous $2$-regular sequence, or sequence in general for that matter, see \cite{UbiquitousThueMorse}. Over the alphabet $\{0,1\}$ the sequence obeys $t(2n)=t(n)$, $t(2n+1)=1-t(n)$, $t(0)=0$. It is bounded (and hence 2-automatic) and is the fixed point of a primitive substitution system so it has been heavily analysed by other methods, making it interesting to compare with the ghost measure. It is well known that the Thue--Morse sequence has a uniformly converging natural density. Thus, for large $N$ the $\mu_{N}$ are essentially just a set of evenly distributed pure point masses, which is another proof that the ghost measure is Lebesgue measure. 

\end{example}

\section{The Lebesgue Decomposition of the Ghost Measure} \label{AffineClassification}

The goal of this paper is to understand the ghost measures of the affine $2$-regular sequences concretely, not just prove they exist in the abstract. First we understand the general character of the measures by determining their Lebesgue decomposition, as in Theorem \ref{THEOREM}. The proof, apart from \textbf{2C} and \textbf{2D}, is contained in this section. In the next section we study the exact shape of the measures.

\subsection{Proof of (most of) Theorem \ref{THEOREM}}

As a reminder, Theorem \ref{THEOREM} works by breaking the affine $2$-regular sequences up into the seven cases.

\begin{enumerate}
    \item The homogeneous case, $b_{0}=b_{1}=0$:
    
    \textbf{1A}: $A_{0}=A_{1}\neq 0$, \textbf{1B}: $A_{0}\neq A_{1}$, neither $0$, \textbf{1C}: $A_{0}$ or $A_{1}$ is $0$,
    \item The inhomogeneous case, $b_{0}+b_{1}\neq 0$:
    
    \textbf{2A}: $A_{0}+A_{1}\leqslant 2$, \textbf{2B}: $A_{0}=A_{1}>1$, \textbf{2C}: $A_{0}\neq A_{1}$, neither $0$, \textbf{2D}: $A_{0}$ or $A_{1}=0$, the other greater than or equal to $3$.
\end{enumerate}

The proof mostly consists of analysing the Fourier coefficients of the $\mu_{N}$. This is the tactic used in \cite{SternSequenceMeasure} and it suffices to prove the theorem apart from \textbf{2C} and \textbf{2D}. We begin by determining $\Sigma(N)=\sum_{n=0}^{2^{N}-1}f(2^{N}+n)$.

\begin{lem}
Let $A=A_{0}+A_{1}$ and $b=b_{0}+b_{1}$. Then
$$
\Sigma(N)=\begin{cases}
b2^{N-1} & A=0 \footnotemark\\
f(1)+b(2^{N}-1) & A=1\\
2^{N}f(1)+b\cdot N2^{N-1} & A=2\\
A^{N}f(1)+b\frac{A^{N}-2^{N}}{A-2} & A>2
\end{cases}.
$$
\footnotetext{Except possibly for $N=0$, depending on the value chosen for $f(1)$.}
\end{lem}
\begin{proof}
We split the sum defining $\Sigma(N)$ by odd and even terms and apply $f(n)$'s defining relations to get $\Sigma(N) = A \cdot \Sigma(N-1)+b2^{N-1}$. Using induction, this says that $\Sigma(N)=A^{N}\Sigma(0)+b(2^{N-1}+2^{N-2}A + \cdots + 2A^{N-2}+A^{N-1})$. This immediately reduces to the required form, using $\Sigma(0)=f(1)$ and $x^{n-1}+x^{n-2}y+\cdots xy^{n-2} +y^{n-1} =\frac{x^{n}-y^{n}}{x-y}$ for $A>2$.
\end{proof}

It is convenient to factor out the $A^{N}$ from $\Sigma(N)$ and call the remainder $\sigma(N)$.
$$
\sigma(N)=\begin{cases}
f(1)+b(2^{N}-1) & A=1\\
f(1)+b\cdot \frac{N}{2} & A=2\\
f(1)+b\frac{1-(\frac{2}{A})^{N}}{A-2} & A>2.
\end{cases}
$$
When $A>2$, we label $\lim_{N \to \infty}\sigma(N) = f(1)+\frac{b}{A-2}$ by $\sigma(\infty)$.

\medskip

Given a finite Borel measure $\mu$ on $\mathbb{T}$, its $m^{th}$ Fourier coefficient, $m \in \mathbb{Z}$, is
$$
\widehat{\mu}(m)=\int_{\mathbb{T}}e^{-2 \pi i m x}d\mu(x).
$$
For absolutely continuous measures $d\mu=g d\lambda$ this reduces to the standard definition of the Fourier coefficients of the function $g$. The key theorem that we need is the following.

\begin{thm}[Levy's Continuity Theorem] \label{Convergence}\cite[Theorem 3.14]{BergForstPotentialTheory} 
A sequence of probability measures $\mu_{n}$ on $\mathbb{T}$ vaguely converges to $\mu$ if and only if for all $m \in \mathbb{Z}$, $\lim_{n \to \infty}\widehat{\mu_{n}}(m)$ exists and equals $\widehat{\mu}(m)$. These coefficients uniquely determine $\mu$.
\end{thm}

We will need to know the Fourier coefficients of a Dirac delta measure;
\begin{align} \label{DeltaCoefficients}
\widehat{\delta_{x}}(m)=e^{-2 \pi i m x}.
\end{align}
The standard convolution theorem for Fourier coefficients shall also be useful
$$
\widehat{\mu * \nu}(m)=\widehat{\mu}(m)\widehat{\nu}(m).
$$
Using the definition of convolution of measures, or by comparing Fourier transforms and using Theorem \ref{Convergence}, we obtain $\delta_{x}*\delta_{y}=\delta_{x+y}$, where $x+y$ is modulo~$1$.

Our first task is to prove that the ghost measures of Theorem \ref{THEOREM} actually exist.

\begin{lem} \label{EXISTENCE}
The ghost measure $\mu$ exists for the sequences under consideration in Theorem \ref{THEOREM}.
\end{lem}
\begin{proof}

We ignore the case $A_{0}+A_{1}=0$. Splitting by odd and even $n$ again gives
\begin{align*}
\mu_{N}&=\frac{1}{\Sigma(N)}\sum_{n=0}^{2^{N-1}-1}f(2^{N}+2n)\delta_{\frac{2n}{2^{N}}}+f(2^{N}+2n+1)\delta_{\frac{2n}{2^{N}}}*\delta_{\frac{1}{2^{N}}}\\
&=\frac{1}{\Sigma(N)}\sum_{n=0}^{2^{N-1}-1}(A_{0}f(2^{N-1}+n)+b_{0})\delta_{\frac{n}{2^{N-1}}}+(A_{1}f(2^{N-1}+n)+b_{1})\delta_{\frac{2n}{2^{N}}}*\delta_{\frac{1}{2^{N}}}\\
&=\frac{\Sigma(N-1)}{\Sigma(N)} \mu_{N-1} *(A_{0}\delta_{0}+A_{1}\delta_{\frac{1}{2^{N}}})+\frac{1}{\Sigma(N)}\left( \sum_{n=0}^{2^{N-1}-1}\delta_{\frac{n}{2^{N-1}}} \right)*(b_{0}\delta_{0}+b_{1}\delta_{\frac{1}{2^{N}}})\\
&=\frac{\sigma(N-1)}{\sigma(N)} \mu_{N-1} *\frac{A_{0}\delta_{0}+A_{1}\delta_{\frac{1}{2^{N}}}}{A_{0}+A_{1}}+\frac{1}{\sigma(N)}\left( \sum_{n=0}^{2^{N-1}-1}\delta_{\frac{n}{2^{N-1}}} \right)*\frac{b_{0}\delta_{0}+b_{1}\delta_{\frac{1}{2^{N}}}}{(A_{0}+A_{1})^{N}}.
\end{align*}
Next, we use Equation \ref{DeltaCoefficients} and linearity to calculate the Fourier coefficients of $\mu_{N}$. Only integer values of $t$ are relevant, so the sum term in the previous line turns into a multiple of the indicator function on $2^{N-1}\mathbb{Z}$;
$$\widehat{\mu_{N}}(t)=\frac{\sigma(N-1)}{\sigma(N)} \widehat{\mu_{N-1}}(t) \frac{A_{0}+A_{1}e^{\frac{-2 \pi i t}{2^{N}}}}{A_{0}+A_{1}}+\frac{1}{\sigma(N)}2^{N-1}\chi_{2^{N-1}\mathbb{Z}}(t)\frac{b_{0}+b_{1}e^{\frac{-2 \pi i t}{2^{N}}}}{(A_{0}+A_{1})^{N}}.$$
By induction we get the general formula

\begin{align} \label{GeneralFormula1}
\widehat{\mu_{N}}(t)&=\frac{\sigma(0)}{\sigma(N)} \prod_{n=1}^{N} \frac{A_{0}+A_{1}e^{\frac{-2 \pi i t}{2^{n}}}}{A_{0}+A_{1}}  \\
& +\frac{1}{\sigma(N)}\sum_{n=1}^{N}\left(2^{n-1}\chi_{2^{n-1}\mathbb{Z}}(t) \frac{b_{0}+b_{1}e^{\frac{-2 \pi i t}{2^{n}}}}{(A_{0}+A_{1})^{n}}\prod_{j=n+1}^{N} \frac{A_{0}+A_{1}e^{\frac{-2 \pi i t}{2^{j}}}}{A_{0}+A_{1}} \right). \nonumber
\end{align}

Taking the limit $N \to \infty$, we obtain

\begin{align} \label{GeneralFormula2}
\widehat{\mu_{N}}(t)&=\frac{\sigma(0)}{\sigma(\infty)} \prod_{n=1}^{\infty} \frac{A_{0}+A_{1}e^{\frac{-2 \pi i t}{2^{n}}}}{A_{0}+A_{1}}  \\
& +\frac{1}{\sigma(\infty)}\sum_{n=1}^{\infty}\left(2^{n-1}\chi_{2^{n-1}\mathbb{Z}}(t) \frac{b_{0}+b_{1}e^{\frac{-2 \pi i t}{2^{n}}}}{(A_{0}+A_{1})^{n}}\prod_{j=n+1}^{N} \frac{A_{0}+A_{1}e^{\frac{-2 \pi i t}{2^{j}}}}{A_{0}+A_{1}} \right). \nonumber
\end{align}

These limits actually exist. For $t=0$, $\widehat{\mu_{N}}(0)=1$ for all $N$ in all cases, essentially by definition of $\Sigma(N)$. So the limit exists at $t=0$, and $\widehat{\mu}(0)=1$, so $\mu$ is always a probability measure. For $t\neq 0$, there will be an $N$ such that $t \notin 2^{N-1}\mathbb{Z}$ and the sum truncates at this finite position. All of the products appearing have terms with magnitude less than $1$, and tending to $1$ as the index of the product increases, so they all converge. Then, because $\frac{1}{\sigma(N)}$ also converges (to a positive number or $0$), we see that $\widehat{\mu_{N}}(t)$ converges for all $t$. By Theorem \ref{Convergence}, the $\mu_{N}$ are vaguely converging to a measure $\mu$ with coefficients $\widehat{\mu}(t)$.
\end{proof}

\bigskip

We now treat the cases \textbf{1A}, \textbf{1B}, \textbf{1C}, \textbf{2A}, \textbf{2B} and part of \textbf{2C} of Theorem \ref{THEOREM} in a series of propositions.

\begin{prop}[\textbf{1A}]
If $A_{0}=A_{1}=A\neq 0$ and $b_{0}=b_{1}=0$, then $\mu$ is equal to Lebesgue measure.
\end{prop}
\begin{proof}
In this case $f(n)$ equals $f(1)\cdot A^{N}$ between $2^{N}$ and  $2^{N+1}-1$. Hence for every $N$, $\mu_{N}$ equals the evenly spread Dirac comb $\frac{1}{2^{N}} \sum_{n=0}^{2^{N}-1} \delta_{\frac{n}{2^{N}}}$. This sequence vaguely converges to Lebesgue measure. One way to show this is to observe that $\widehat{\mu_{N}} = \chi_{2^{N}\mathbb{Z}}(t)$, which converges pointwise to $\chi_{\{0\}}=\widehat{\lambda}$.
\end{proof}

\begin{prop}[\textbf{1B}]
If $A_{0}\neq A_{1}$, neither equal $0$, and $b_{0}=b_{1}=0$, then $\mu$ is singular continuous.
\end{prop}
\begin{proof}
The proof that the ghost measure of Stern's sequence is singular continuous works, almost without alteration. One can read \cite{SternSequenceMeasure} for the full details. The proof relies on the representation of $\mu$ and its Fourier transforms by convergent infinite products (both of which follow from the formulas of Lemma \ref{EXISTENCE}) 
$$
\mu=\Conv_{n=1}^{\infty}\frac{A_{0}\delta_{0}+A_{1}\delta_{\frac{1}{2^{n}}}}{A_{0}+A_{1}},\quad \widehat{\mu}=\prod_{n=1}^{\infty}\frac{A_{0}+A_{1}e^{\frac{-2 \pi i t}{2^{n}}}}{A_{0}+A_{1}}.
$$

Because $\mu$ is represented by a convergent infinite convolution product of pure point measures, it has pure type. That is, only one of $\mu_{ac},\mu_{sc},\mu_{pp}$ is non-zero \cite[Theorem 35]{BorgeJessenDistributionsAndZetaFunction}. To prove that it is not absolutely continuous, one first shows that $\widehat{\mu}(1)\neq 0$ and combines with $\widehat{\mu}(t)=\widehat{\mu}(2t)$ for all $t\in \mathbb{Z}$ (both easy results) to show that infinitely many coefficients have the same non-zero value. However, the Riemann-Lebesgue lemma states that the Fourier coefficients of an absolutely continuous measure limits to zero as $\abs{t}$ tends to infinity. 

The only change to the proof that $\mu$ is not pure point is that the original has 

$$
\frac{\max_{t \in [0,\frac{2}{5}]}(\abs{\widehat{\mu}(1-t)}^{2})}{\min_{t \in [0,\frac{2}{5}]}(\abs{\widehat{\mu}(t)}^{2})}=\kappa<\frac{1}{16},
$$
while we can only prove that $\kappa<1$ because $\abs{\widehat{\mu}(t)}^{2}=\prod_{k=1}^{\infty}\frac{A_{0}^{2}+A_{1}^{2}+2A_{0}A_{1}\cos(\frac{\pi t}{2^{k}})}{(A_{0}+A_{1})^{2}}$ is strictly decreasing in $[0,1]$. However, this is sufficient for our needs. The proof in \cite{SternSequenceMeasure} defines $\sigma_{N}=\frac{1}{2^{N}}\sum_{n=1}^{2^{N}}\abs{\widehat{\mu}(n)}^{2}$, and then proves that (for $N>2$) $\sigma_{N}\leqslant  \sigma_{N-1}-\frac{1-\kappa}{4}\sigma_{N-2}$. Substituting for $\sigma_{N-1}$ gives $\sigma_{N} \leqslant \frac{3+\kappa}{4}\sigma_{N-2}$. Inductively, we then have $\sigma_{N} \leqslant \left(\frac{3+\kappa}{4}\right)^{\frac{N}{2}}\max(\sigma_{0},\sigma_{1})$, which tends to $0$ as $N$ tends to infinity because $\kappa<1$. This is sufficient to prove the result, by using Wiener's criterion \cite[Proposition 8.9]{baake_grimm_2017}, that $\sigma_{N}$ tends to zero as $N$ tends to infinity exactly when the measure has no pure point part.
\end{proof}

\begin{prop}[\textbf{1C}]
If $A_{0}$ or $A_{1}$ is $0$ and $b_{0}=b_{1}=0$, then $\mu$ is pure point, and equals $\delta_{0}$.
\end{prop}
\begin{proof}
If $A_{0}\neq 0$ and $A_{1}=0$, then $f(n)=0$ except at $n=2^{N}$ where it equals $f(1)\cdot A_{0}^{N}$. Then all $\mu_{N}$ are equal to $\delta_{0}$, and hence so is $\mu$. Similarly if $A_{0}=0$ and $A_{1} \neq 0$, the sequence is only non zero at $n=2^{N}-1$. So $\mu_{N}=\delta_{\frac{2^{N}-1}{2^{N}}}=\delta_{-\frac{1}{2^{N}}}$ which vaguely converges to $\delta_{0}$ as well.
\end{proof}

\begin{prop}[\textbf{2A}]
If $A_{0}+A_{1}\leqslant 2$ and $b_{0}+b_{1}\neq 0$, then $\mu=\lambda$.
\end{prop}
\begin{proof}
First consider $A_{0}=A_{1}=0$. Then $f(n)=b_{0}$ if $n$ is even, and $b_{1}$ if $n$ is odd (except possibly for $f(1)$). Thus (forgetting about $N=0$)
$$\mu_{N}=\frac{1}{2^{N-1}(b_{0}+b_{1})}\left(\sum_{n=0}^{2^{N-1}-1} \delta_{\frac{n}{2^{N-1}}}\right)*(b_{0}\delta_{0}+b_{1}\delta_{\frac{1}{2^{N}}}).$$
The Fourier transform is
$$
\widehat{\mu_{N}}(t)=\frac{b_{0}+b_{1}e^{-\frac{2 \pi i t }{2^{N}}}}{b_{0}+b_{1}}\chi_{2^{N-1}\mathbb{Z}}(t).
$$
For any finite $t \neq 0$, the coefficients converge to zero. For $t=0$, $\widehat{\mu_{N}}(0)$ is always $1$. Hence $\widehat{\mu_{N}}$ converges to $\chi_{\{0\}}$, and hence $\mu_{N}$ vaguely converges to Lebesgue measure. For $A_{0}+A_{1}=1$ or $2$, take Equation \ref{GeneralFormula1}. If $t=2^{a}b$ where $b$ is odd, the sum truncates at $N=a+1$. In particular, $\widehat{\mu_{N}}$ has magnitude bounded by

$$
\abs{\widehat{\mu_{N}}(t)} \leqslant \frac{1}{\sigma(N)} \left(\sigma(0) +\sum_{n=1}^{a+1}2^{n-1} \frac{b_{0}+b_{1}}{(A_{0}+A_{1})^{n}}\right).
$$
The bracketed term is eventually independent of $N$. Since $\sigma(N)$ tends to infinity with $N$, $\widehat{\mu}(t)=0$ for all $t\neq 0$. So the $\mu_{N}$ converge to Lebesgue measure. 
\end{proof}

\begin{prop}[\textbf{2B}]
If $A_{0}=A_{1}>1$ and $b_{0}+b_{1}\neq 0$, then $\mu$ is absolutely continuous.
\end{prop}
\begin{proof}
Denote by $A$ the common value of $A_{0}$ and $A_{1}$. In this instance, it is easier to start with Equation \ref{eqn1}. Specialising to the current case gives

$$\widehat{\mu_{N}}(t)=\frac{\sigma(N-1)}{\sigma(N)} \widehat{\mu_{N-1}}(t) \frac{1+e^{\frac{-2 \pi i t}{2^{N}}}}{2}+\frac{1}{\sigma(N)}2^{N-1}\chi_{2^{N-1}\mathbb{Z}}(t)\frac{b_{0}+b_{1}e^{\frac{-2 \pi i t}{2^{N}}}}{(2A)^{N}} $$
Write $t=2^{a}b$, $b$ odd. At $N=a+1$, 
$$\frac{1+e^{\frac{-2 \pi i t}{2^{N}}}}{2}=\frac{1+e^{\frac{-2 \pi i 2^{a}b}{2^{a+1}}}}{2}=\frac{1+e^{-\pi i b}}{2}=0.$$ 
Hence at $N=a+1$ the $\widehat{\mu_{N-1}}$ term in the formula for $\widehat{\mu_{N}}$ vanishes, leaving only $\widehat{\mu_{a+1}}(2^{a}b)=\frac{1}{2\sigma(a+1)}\frac{b_{0}-b_{1}}{A^{a+1}}$. Moreover, for larger $N$, $2^{a}b \notin 2^{N-1}\mathbb{Z}$ and so the second term in the recursion is zero, and only the product term contributes, resulting in
$$\widehat{\mu_{a+n}}(2^{a}b)=\frac{b_{0}-b_{1}}{2\sigma(a+n)}\frac{1}{A^{a+1}}\prod_{j=1}^{n-1}\frac{1+e^{\frac{-\pi i b}{2^{j}}}}{2} \to \widehat{\mu}(2^{a}b)= \frac{b_{0}-b_{1}}{2\sigma(\infty)}\frac{1}{A^{a+1}}\prod_{j=1}^{\infty}\frac{1+e^{\frac{-\pi i b}{2^{j}}}}{2}.
$$

To show that $\mu$ is absolutely continuous, we use the Riesz--Fischer Theorem \cite[Theorem 4.17]{RudinAnalysis}, which states that if $\sum_{n=-\infty}^{\infty}\abs{c_{n}}^{2}<\infty$, then there is a $g$ in $L^{2}(\mathbb{T})\subset L^{1}(\mathbb{T})$ such that $c_{n}=\widehat{g}(n)$ for all $n \in \mathbb{Z}$. Hence, if the Fourier coefficients of $\mu$ are square summable, $\mu$ has the same Fourier coefficients as $g d \lambda$. The coefficients uniquely determine $\mu$ so $d\mu=gd\lambda$ and $\mu$ is absolutely continuous. Here, the square absolute values of $\mu$'s coefficients are, for $t =2^{a}b \neq 0$,
$$\abs{\widehat{\mu}(2^{a}b)}^{2}=\frac{1}{4\sigma(\infty)^{2}}\frac{(b_{0}-b_{1})^{2}}{A^{2a+2}}\prod_{j=1}^{\infty}\frac{\abs{1+e^{\frac{-\pi i b}{2^{j}}}}^{2}}{2^{2}}=\frac{1}{4\sigma(\infty)^{2}}\frac{(b_{0}-b_{1})^{2}}{A^{2a+2}}\prod_{j=2}^{\infty}\cos(\frac{\pi b}{2^{j}})^{2}.$$
In light of  $\abs{\widehat{\mu}(t)}^{2} = \abs{\widehat{\mu}(-t)}^{2}$, it suffices to prove $\sum_{t=1}^{\infty}\abs{\widehat{\mu}(t)}^{2}<\infty$. Each $t>0$ can be written uniquely in the form $t=2^{a}b$, so we can break up the sum to obtain
\begin{align*}
\sum_{n=1}^{\infty}\abs{\widehat{\mu}(t)}^{2}&=\frac{(b_{0}-b_{1})^{2}}{4\sigma(\infty)^{2}}\sum_{a=0}^{\infty}\sum_{c=0}^{\infty}\frac{1}{A^{2a+2}}\prod_{j=2}^{\infty}\cos(\frac{\pi (2c+1)}{2^{j}})^{2}\\
&=\frac{(b_{0}-b_{1})^{2}}{4\sigma(\infty)^{2}}\frac{1}{A^{2}-1} \left( \sum_{c=0}^{\infty} \prod_{j=1}^{\infty}\cos\left(\frac{\pi (c+\frac{1}{2})}{2^{j}}\right)^{2}\right).
\end{align*}
Using Viete's identity, $\frac{\sin(x)}{x}=\prod_{j=1}^{\infty}\cos(\frac{x}{2})$, we have
$$
\sum_{c=0}^{\infty} \prod_{j=1}^{\infty}\cos(\frac{\pi (c+\frac{1}{2})}{2^{j}})^{2}= \sum_{c=0}^{\infty} \frac{\sin(\pi (c+\frac{1}{2}))^{2}}{(\pi (c+\frac{1}{2}))^{2}}=\sum_{c=0}^{\infty} \frac{4}{\pi^{2}}\frac{1}{(2c+1)^{2}}=
\frac{1}{2}.
$$
Thus $\sum_{n=1}^{\infty}\abs{\widehat{\mu}(t)}^{2}<\infty$ and $\mu$ is absolutely continuous. This also shows $g=\frac{d \mu}{d \lambda} \in L^{2}(\mathbb{T})$ with $\norm{g}_{2}^{2}=1+\frac{(b_{0}-b_{1})^{2}}{4\sigma(\infty)^{2}}\frac{1}{A^{2}-1}
$.\end{proof}

\begin{prop}[\textbf{2C}]
If $A_{0}\neq A_{1}$, neither $0$ and $b_{0}+b_{1}\neq 0$, then $\mu$ is continuous.
\end{prop}
\begin{proof}
We start with the general formula in Equation \ref{eqn1}. For $t\neq 0$, write $t=2^{a}b$. As before, the sum truncates at $n=a+1$, resulting in
$$
\abs{\widehat{\mu}(2^{a}b)}=\frac{1}{\sigma(\infty)} \abs{\prod_{n=1}^{\infty} \frac{A_{0}+A_{1}e^{\frac{-2 \pi i 2^{a}b}{2^{n}}}}{A_{0}+A_{1}} \left( \sigma(0) +\sum_{n=1}^{a+1} 2^{n-1} \frac{b_{0}+b_{1}e^{\frac{-2\pi i 2^{a}b}{2^{n}}}}{\prod_{j=1}^{n}A_{0}+A_{1}e^{\frac{-2\pi i 2^{a}b}{2^{j}}}}  \right)}.
$$
Here, we recognise the product pulled out in front as the same coefficients appearing in case 1B; call this term $\widehat{\nu}$. The sum term in the brackets is uniformly bounded above by some constant $K$ for all values of $2^{a}b$. Thus $\abs{\widehat{\mu}(t)} \leqslant K \cdot \abs{\widehat{\nu}(t)}$ for all $t$. Hence, $\frac{1}{2^{N}}\sum_{n=1}^{2^{N}} \abs{\widehat{\mu}(t)}^{2} \leqslant K^{2}\frac{1}{2^{N}}\sum_{n=1}^{2^{N}}\abs{\widehat{\nu}(t)}^{2}$. This tends to $0$ as $N$ tends to infinity, proving that $\mu$ has no pure point part, by the proof of \textbf{1B}.
\end{proof}

\section{Finding the Shape of the Measures} \label{SHAPE}

In this section, we obtain more detailed information about the \textbf{2B}, \textbf{2C} and \textbf{2D} measures. For \textbf{2B} and \textbf{2D} we can explicitly describe the measure by determining its Radon--Nikodym derivative in the first case, and by specifying all of its pure point parts in the second (this also suffices to prove the measure is pure point). For \textbf{2C}, we complete the proof of singular continuity, and also give information about where it is concentrated. 

The distribution function of a measure $\mu$ on $\mathbb{T}= [0,1)$ is the function $F(x)=\mu([0,x])$. If $\mu$ has the decomposition $d \mu=g d\lambda +d\mu_{s}$ ($g d \lambda$ the absolutely continuous part, $d \mu_{s}$ the singular part), then it is a classical result \cite[Chapter 7]{RudinAnalysis} that $F^{\prime}(x)=g(x)$ for Lebesgue almost all $x$. Thus, one can calculate $g(x)$, using $F^{\prime}(x)=\lim_{h \to 0} \frac{F(x+h)-F(x)}{h} = \lim_{h \to 0}\frac{\mu((x,x+h])}{\lambda((x,x+h])}$ assuming $h>0$, with a similar expression if $h<0$. It is convenient to be able to generalise the notion of `differentiation' that one uses. The following definition and theorem can be found in \cite{RudinAnalysis}. They can easily be modified to apply to $\mathbb{T}$ rather than $\mathbb{R}^{k}$. The notation $B(x,r)$ denotes the open ball with centre $x$ and radius $r$.

\begin{dfn} \cite[Definition 7.9]{RudinAnalysis}
A sequence of sets $\{E_{i}(x)\}$ in $\mathbb{R}^{k}$ is nicely shrinking to $x\in \mathbb{R}^{k}$ if there exists an $\alpha>0$ and a sequence $r_{i}\to 0$ as $i\to \infty$ such that $E_{i}(x)\subset B(x,r_{i})$ and $\lambda(E_{i}(x))>\alpha \lambda(B(x,r_{i}))$ for all $i$, where $\lambda$ is Lebesgue measure. The $E_{i}(x)$ need not contain $x$, or obey any other conditions.
\end{dfn}

\begin{thm} \label{RudinTheorem}
\cite[Theorem 7.14]{RudinAnalysis}. Suppose $\mu$ is a finite Borel measure on $\mathbb{R}^{k}$ with Lebesgue decomposition $d\mu=g d\lambda + d \mu_{s}$. For each $x$ choose any sequence $\{E_{i}(x)\}$ which nicely shrinks to $x$. Then for $\lambda$-almost every $x$
$$
\lim_{i\to \infty}\frac{\mu(E_{i}(x))}{\lambda(E_{i}(x))}=g(x).
$$
\end{thm}

Let $x=(0.x_{1}x_{2}\cdots)_{2}$ be the binary representation of $x \in [0,1)$. We always choose the representation with infinitely many trailing $0$'s if $x$ is a dyadic rational. We take $E_{i}(x)$ to be the half open interval $[(0.x_{1}\cdots x_{i}00\cdots)_{2},(0.x_{1}\cdots x_{i}11\cdots)_{2})$. 
Then $E_{i}(x)$ has length $\frac{1}{2^{i}}$, and so $E_{i}(x)\subset B(x,\frac{1}{2^{i}})$ and $\lambda(E_{i}(x))\geqslant \frac{1}{2}\lambda( B(x,\frac{1}{2^{i}}))$. Hence the $E_{i}(x)$ are nicely shrinking to $x$. Our work will require the Portmanteau Theorem, stated below as it applies to our situation.

\begin{thm} \label{Portmanteau}
Let $\mu$, $\{\mu_{n}\}$ be Borel probability measures on $\mathbb{T}$. The following are equivalent.
\begin{enumerate}[(i)]
    \item $\mu_{n}$ converges vaguely to $\mu$.
    \item $\limsup_{n\to \infty} \mu_{n}(F)\leqslant \mu (F)$ for all closed sets $F$. \label{Second}
    \item $\liminf_{n\to \infty} \mu_{n}(G)\geqslant \mu (G)$ for all open sets $G$.
    \item $\lim_{n\to \infty}\mu_{n}(A)=\mu(A)$ for all sets $A$ such that $\mu(\partial A)=0$. \label{Fourth}
\end{enumerate}
\end{thm}

We are almost ready to give an explicit formula for the Radon--Nikodym derivative of the ghost measure of a sequence in Case \textbf{2B} of Theorem \ref{THEOREM}. We require which also applies to \textbf{2C}.

\begin{lem} \label{MainResult}
Assume $A_{0},A_{1}>0$ and $A_{0}+A_{1}\geqslant 3$. Let 
$x=(0.x_{1}x_{2}\ldots)_{2}$. Then
\begin{align*}
    \mu(E_{i}(x))=\frac{1}{(A_{0}+A_{1})^{i}}\frac{1}{\sigma(\infty)}\left(A_{x_{1}}\cdots A_{x_{i}}\left(f(1)+\sum_{j=1}^{i}\frac{b_{x_{j}}}{\prod_{k=1}^{j}A_{x_{k}}}\right)+\frac{(b_{0}+b_{1})}{A_{0}+A_{1}-2}\right).
\end{align*}
\end{lem}
\begin{proof}
Assume $N>i$ and put $N=i+c$. Then $\frac{n}{2^{N}}\in E_{i}(x)$ if and only if $n=x_{1}\cdots x_{i}0^{c} +r$ for some $0 \leqslant r \leqslant 2^{c}-1$. Let $r=(r_{1}\ldots r_{c})_{2}$. Then
$$
f((1x_{1}\cdots x_{i}r_{1}\cdots r_{c})_{2})= A_{r_{c}}\cdots A_{r_{1}}f((1x_{1}\cdots x_{i})_{2})+\sum_{j=1}^{c}\left(\prod_{l=j+1}^{c}A_{r_{l}}\right)b_{r_{j}}.
$$
With this, the value of $\mu_{N}(E_{i}(x)) = \frac{1}{\Sigma(N)}\sum_{\frac{n}{2^{N}}\in E_{i}(x)}f(2^{N}+n)$ satisfies 

\begin{align*}
\mu_{i+c}(E_{i}(x))&=\frac{1}{\Sigma(i+c)}\sum_{r=0}^{2^{c}-1}A_{r_{c}}\cdots A_{r_{1}}f((1x_{1}\cdots x_{i})_{2})+\sum_{j=1}^{c}\left(\prod_{l=j+1}^{c}A_{r_{l}}\right)b_{r_{j}}\\
&=\frac{1}{\Sigma(i+c)}\left(f((1x_{1}\cdots x_{i})_{2})(A_{0}+A_{1})^{c}+\sum_{r=0}^{2^{c}-1}\sum_{j=1}^{c}\left(\prod_{l=j+1}^{c}A_{r_{l}}\right)b_{r_{j}}\right)\\
&=\frac{1}{(A_{0}+A_{1})^{i}}\frac{1}{\sigma(i+c)}\left(f((1x_{1}\cdots x_{i})_{2})+\frac{(b_{0}+b_{1})\left(1-\frac{2}{A_{0}+A_{1}}^{c}\right)}{A_{0}+A_{1}-2}\right).
\end{align*}
Next take $c\to \infty$ and apply Theorem \ref{Portmanteau} \eqref{Fourth}, since $\mu$ is continuous, to obtain

$$
\mu(E_{i}(x))=\frac{1}{(A_{0}+A_{1})^{i}}\frac{1}{\sigma(\infty)}\left(f((1x_{1}\cdots x_{i})_{2})+\frac{(b_{0}+b_{1})}{A_{0}+A_{1}-2}\right).
$$
Now use $f((1x_{1}\cdots x_{i})_{2})=A_{x_{1}}\cdots A_{x_{i}}\left(f(1)+\sum_{j=1}^{i}\frac{b_{x_{j}}}{\prod_{k=1}^{j}A_{x_{k}}}\right)$.
\end{proof}

\begin{thm}
In \textbf{2B}, $A_{0}=A_{1}=A>1$, $\mu$ has Radon--Nikodym derivative
$$
g(x)=\frac{f(1)+\sum_{j=1}^{\infty}\frac{b_{x_{j}}}{A^{j}}}{f(1)+\frac{b_{0}+b_{1}}{2A-2}},
$$
where $x=(0.x_{1}x_{2}\cdots)_{2}$ is the binary representation of $x$.
\end{thm}
\begin{proof}

Take Lemma \ref{MainResult}, insert $A=A_{0}=A_{1}$ and $\lambda(E_{i}(x))=\frac{1}{2^{i}}$ to get $\frac{\mu(E_{i}(x))}{\lambda(E_{i}(x))} = \frac{1}{\sigma(\infty)} ( f(1) + \sum_{j=1}^{i}  \frac{b_{x_{j}}}{A^{j}} + \mathcal{O}(\frac{1}{A^{i}}) )$. Taking $i \to \infty$, using the value of $\sigma(\infty)$ and Theorem \ref{RudinTheorem} give the result. 
\end{proof}

We now show that any Case \textbf{2C} measure is singular continuous, by proving that the Radon--Nikodym derivative of its absolutely continuous part is zero. This also provides information about where $\mu$ is concentrated.

\begin{lem} \label{LemmaName}
If $0<A_{0} < A_{1}$, then $\lim_{i\to \infty}\mu(E_{i}(x)) / \lambda(E_{i}(x))=0$ if the lower density of $0$'s in $x$'s binary representation $x=(0.x_{1}x_{2}\cdots)_{2}$ is greater than $\Lambda := \log(\frac{2A_{1}}{A_{0}+A_{1}}) / \log(\frac{A_{1}}{A_{0}})$. The result for $A_{1}<A_{0}$ is the same with $0$ and $1$ swapped.
\end{lem}
\begin{proof}

Let $\overline{A}=\frac{A_{0}+A_{1}}{2}$, and $\#0_{i}$ be the number of $0$'s in $\{x_{1},\ldots,x_{i}\}$ and similarly for $\#1_{i}$. Let $\Lambda_{i}=\frac{\#0_{i}}{i}$ be the fraction of zeroes in that set. In Lemma \ref{MainResult}, the sum term is at most $\mathcal{O}(i)$ if either $A_{i}$ is $1$, and is $\mathcal{O}(1)$ otherwise, so upon taking logarithms one arrives at

\begin{align*}
\log \left( \frac{\mu(E_{i}(x))}{\lambda(E_{i}(x))} \right)&=\#0_{i} \cdot \log \left(\frac{A_{0}}{\overline{A}}\right)+\#1_{i} \cdot \log \left(\frac{A_{1}}{\overline{A}}\right)+\mathcal{O}(\log(i))\\
&=i \cdot \left(\Lambda_{i} \cdot \log \left(\frac{A_{0}}{\overline{A}}\right)+(1-\Lambda_{i}) \cdot \log \left(\frac{A_{1}}{\overline{A}}\right)\right)+\mathcal{O}(\log(i)).
\end{align*}
If one takes $\Lambda$ as in the statement of the lemma, uses $A_{0}<A_{1}$ then the quantity in brackets is negative if $\Lambda_{i}>\Lambda$. Thus, the expression diverges to negative infinity if $\Lambda_{i}>\Lambda+\varepsilon$, for some $\varepsilon>0$ and all sufficiently large $i$. Thus $\lim_{i \to \infty}\frac{\mu(E_{i}(x))}{\lambda(E_{i}(x))}=0$, if the lower density of $0$'s in $x$'s binary expansion is greater than $\Lambda$.  \end{proof}

\begin{thm} 
In \textbf{2C}, the ghost measure is singular continuous.
\end{thm}
\begin{proof}
We already know that $\mu$ has no pure point part. Theorem \ref{RudinTheorem} says that  $\lim_{i \to \infty}\frac{\mu(E_{i}(x))}{\lambda(E_{i}(x))}=g(x)$ for Lebesgue almost all $x$. One can show that $\Lambda<\frac{1}{2}$, so Lemma \ref{LemmaName} implies that whether $A_{0}<A_{1}$ or $A_{0}>A_{1}$, this limit is zero if $x$ has density $\frac{1}{2}$ of zeroes, i.e. if $x$ is normal in base $2$. This is Lebesgue almost all $x\in [0,1)$ by Borel's Normal Number theorem \cite{BorelNormalNumbers}. Hence $g$ is the zero function, and $\mu$ has no absolutely continuous part. Thus $\mu$ is singular continuous.
\end{proof}

In principal, Lemma \ref{MainResult} allows us to calculate the measure of any interval, while the behaviour of $\frac{\mu(E_{i}(x))}{\lambda(E_{i}(x))}$ as $i\to \infty$ indicates how $\mu$'s mass is arranged around $x$. We can also say something about where $\mu$ is concentrated.

\begin{thm}
Suppose $A_{0}<A_{1}$ (resp. $A_{0}>A_{1}$). Then $\mu$ is concentrated on the complement of the set $\mathfrak{R}_{\Lambda} \subset [0,1)$ of all $x$ with lower density of zeroes (resp. ones) greater than $\Lambda$ in their binary representation.
\end{thm}
\begin{proof}
The result \cite[Theorem 7.15]{RudinAnalysis} includes a proof that if $\mu$ is a measure on $\mathbb{T}$ which is singular with respect to $\lambda$, then $ \mathfrak{C}= \{x \in [0,1) \; |\; \exists \; \{r_{i}\}_{i=1}^{\infty} \to 0 \text { as } i \to \infty, \; \exists M>0:\; \forall i \;\frac{\mu(B(x,r_{i}))}{\lambda(B(x,r_{i}))}<M \}$ has $\mu$-measure $0$. Consider $E_{i}(x)$, where $x$ is not a dyadic rational (which we ignore since they are countable and $\mu$ is continuous). Then $\ldots x_{i}01\ldots$ or $\ldots x_{i}10\ldots$ occur infinitely often in $x$'s binary representation. For these $i$, $B(x,\frac{1}{2^{i+2}})\subset E_{i}(x)$. Thus
$$
\frac{\mu(B(x,\frac{1}{2^{i+2}}))}{\lambda(B(x,\frac{1}{2^{i+2}}))} \leqslant 2\cdot \frac{\mu(E_{i}(x))}{\lambda(E_{i}(x))}
$$
Hence if for these $i$ (and in particular if for all $i$), $\mu(E_{i}(x)) / \lambda(E_{i}(x))$ remains bounded then $x\in \mathfrak{C}$. Thus by Lemma \ref{LemmaName} $\mathfrak{R}_{\Lambda} \subset \mathfrak{C}$ and so $\mathfrak{R}_{\Lambda}$ has $\mu$-measure $0$, and $\mu$ is concentrated on its complement, $\mu(\mathfrak{R}_{\Lambda}^{c})=1$.
\end{proof}

\subsection{Case 2D}

\begin{thm} \label{PurePoint}
In Case \textbf{2D} of Theorem \ref{THEOREM}, $\mu$ is pure point and supported on the dyadic rationals.
\end{thm}
\begin{proof}
Assume $A_{1}=0$, and let $A=A_{0}$. The proof for $A_{0}=0$ follows \emph{mutatis mutandis}. Let $x=(0.x_{1}\cdots)_{2}$ be a dyadic rational. We estimate the measure of $\{x\}$. Since singletons are closed, we can use Theorem \ref{Portmanteau} \eqref{Second}. We start with
$$
\mu_{N}(\{x\})=\frac{f((1x_{1}\cdots x_{N})_{2})} {\Sigma(N)}=\frac{1}{\sigma(N)}\left(\frac{A_{x_{1}}\cdots A_{x_{N}}}{A^{N}}f(1)+\sum_{j=1}^{i}\frac{b_{x_{j}}}{A^{N}}\prod_{k=j+1}^{N}A_{x_{k}}\right).
$$
If $x=0$ then all $x_{i}$ are zero, and all $A_{x_{i}}=A\neq0$. This results in
$$
\mu(\{0\})\geqslant \limsup_{N \to \infty}\frac{1}{\sigma(N)} \left(f(1)+\sum_{j=1}^{N}\frac{b_{0}}{A^{j}}\right)=\frac{1}{\sigma(\infty)}\left(f(1)+\frac{b_{0}}{A-1}\right).
$$
If $x$ is not zero, then some $x_{i}$ is $1$. If $x_{n}$ is the last such bit and $N>n$, then
$$
\mu(\{x\}) \geqslant \limsup_{N \to \infty} \frac{1}{A^{N}}\frac{1}{\sigma(N)}\left(\sum_{j=n}^{N}b_{x_{j}}A^{N-j}\right)= \frac{1}{\sigma(\infty)}\left(\frac{b_{1}}{A^{n}}+\frac{b_{0}}{A^{n}}\frac{1}{A-1}\right).
$$
There are $2^{n-1}$ rationals in $[0,1)$ such that $x_{n}$ is the last non-zero bit. Thus the total mass supplied by these is at least $\frac{1}{\sigma(\infty)}\left(\frac{2}{A}\right)^{n}\left(\frac{b_{1}}{2}+\frac{b_{0}}{2}\frac{1}{A-1}\right)$. Summing over $n$ gives $\frac{1}{\sigma(\infty)}\frac{1}{A-2}\left(b_{1}+\frac{b_{0}}{A-1}\right)$. Adding this to the $x=0$ contribution gives
$$1=\mu(\mathbb{T})\geqslant \sum_{x \in \mathbb{Z}[2]}\mu(\{x\}) \geqslant \frac{1}{\sigma(\infty)}\left(f(1)+\frac{b_{0}}{A-1}+ \frac{1}{A-2}\left(b_{1}+\frac{b_{0}}{A-1}\right)\right)=1.$$
Hence all inequalities here are in fact equalities, and all of $\mu$'s mass is contained in these pure points. Hence $\mu$ is pure point and supported on dyadic rationals.\end{proof}

\section{Concluding Remarks} \label{Observations}

In this paper, we obtained detailed information about specific examples of ghost measures. One may wonder if there is a way to obtain this information without as much manual labour. In particular, is there a way to determine the Lebesgue decomposition of the ghost measure of a $k$-regular sequence directly from its linear representation? There may be no hope in general---the $k$-regular sequences can be quite wild. However, the subset of sequences that practitioners care about seems to be better behaved. For example, the Lebesgue type of every affine $2$-regular sequence is pure, and purity of type seems quite common in general. 

Suppose our $k$-regular sequence $f$ had linear representation involving matrices $C_{0},C_{1},\ldots,C_{k-1}$. Let $Q=C_{0}+\cdots+C_{k-1}$, and let $\rho$ be the spectral radius of $Q$. Also let $\rho^{*}$ be the joint spectral radius of the $C_{i}$. For the obvious linear representations of the affine $2$-regular sequences, these quantities are easily calculated. The below table compares the value of $\log_{k}(\frac{\rho}{\rho^{*}})$ to the Lebesgue type of $\mu$.
\begin{table}[h]
    \scriptsize
    \centering
    \begin{tabular}{c|c|c}
    Case & $\log_{2}(\frac{\rho}{\rho^{*}})$ & Lebesgue type\\
    \hline 
    \textbf{1A} & 1 & A.C.\\
    \textbf{1B} & $\log_{2}(\frac{A_{0}+A_{1}}{\max(A_{0},A_{1})})$ & S.C.\\
    \textbf{1C} & 0 & P.P.\\
    \textbf{2A} & ?\footnotemark & A.C.\\
    \textbf{2B} & 1 & A.C. \\
    \textbf{2C} & $\log_{2}(\frac{A_{0}+A_{1}}{\max(A_{0},A_{1})})$ & S.C. \\
    \textbf{2D} & 0 & P.P\\
    \end{tabular}
    \label{EigenvalueComparison}
\end{table}
\footnotetext{Here the quantity is $1$, except for when some $A_{i}=0$ and the other is $2$, when it equals $0$.}

The pattern is clear: $\log_{k}(\frac{\rho}{\rho^{*}})=1$ implies absolutely continuous, between $0$ and $1$ implies singular continuous, and $0$ implies pure point (almost). This quantity $\log_{k}(\frac{\rho}{\rho^{*}})$ appears elsewhere in connection to $k$-regular sequences, for example in \cite{CEMRegularPaper}, where it is proved that under some non-degeneracy conditions, if $\log_{k}(\frac{\rho}{\rho^{*}})>0$ then the ghost measure of the $k$-regular sequence exists and is continuous. Moreover, the distribution function $F(x)=\mu([0,x])$ has H\"older exponent $\alpha$ for all $\alpha<\log_{k}(\frac{\rho}{\rho^{*}})$, and exponent equal to $\log_{k}(\frac{\rho}{\rho^{*}})$ if the matrices $C_{i}$ obey the so-called finiteness condition. This quantity also seems to be related to the Hausdorff dimension of certain fractals arising from $k$-regular sequences \cite{CEFractalPaper}. We leave it to the reader to further explore these connections.

\section*{Acknowledgements}

I wish to thank Michael Coons for supervising my PhD, of which this work is part. I also thank the Commonwealth of Australia for supporting my studies.

\newpage

\bibliographystyle{unsrt}
\bibliography{bibliography}

\begin{thebibliography}{10}

\bibitem{SternSequenceMeasure}
Michael Baake and Michael Coons.
\newblock A natural probability measure derived from stern's diatomic sequence.
\newblock {\em Acta Arithmetica}, 183, 06 2017.

\bibitem{CEMRegularPaper}
Michael Coons, James Evans, and Neil Manibo.
\newblock Beyond substitutions: The dynamics of regular sequences.
\newblock preprint.

\bibitem{TheAnalyst}
George Berkeley.
\newblock {\em The Analyst}.
\newblock Trinity College, 2002.

\bibitem{AutomaticSequences}
Jean-Paul Allouche and Jeffrey Shallit.
\newblock {\em Automatic Sequences: Theory, Applications, Generalizations}.
\newblock Cambridge University Press, 2003.

\bibitem{RingRegularSequences}
Jean-Paul Allouche and Jeffrey Shallit.
\newblock The ring of k-regular sequences.
\newblock {\em Theoretical Computer Science}, 98:163--197, 1992.

\bibitem{RingRegularSequences2}
Jean-Paul Allouche and Jeffrey Shallit.
\newblock The ring of k-regular sequences ii.
\newblock {\em Theoretical Computer Science}, 307:3--29, 2003.

\bibitem{baake_grimm_2017}
Michael Baake and Uwe Grimm.
\newblock {\em Aperiodic Order}.
\newblock Encyclopedia of Mathematics and its Applications. Cambridge
  University Press, 2017.

\bibitem{RudinFourierAnalysis}
Walter Rudin.
\newblock {\em Fourier Analysis on Groups}.
\newblock John Wiley and Sons, Ltd, 2011.

\bibitem{RudinAnalysis}
Walter Rudin.
\newblock {\em Real and complex analysis}.
\newblock New York : McGraw-Hill, 3rd ed edition, 1987.

\bibitem{BinomialCoefficients}
N.~J. Fine.
\newblock Binomial coefficients modulo a prime.
\newblock {\em The American Mathematical Monthly}, pages 589--592, 12 1947.

\bibitem{SalemStrictlyIncreasingSingularFunctions}
R.~Salem.
\newblock On some singular monotonic functions which are strictly increasing.
\newblock {\em Transactions of the American Mathematical Society},
  53(3):427--439, 1943.

\bibitem{Numberphile}
Brady Haran and Daniel Erman.
\newblock The josephus problem.

\bibitem{UbiquitousThueMorse}
Jean-Paul Allouche and Jeffrey Shallit.
\newblock The ubiquitous prouhet-thue-morse sequence.
\newblock In {\em Sequences and their Applications}, pages 1--16. Springer
  London, 1999.

\bibitem{BergForstPotentialTheory}
Christian Berg and Gunnar Forst.
\newblock {\em Potential Theory on Locally Compact Abelian Groups}.
\newblock Springer-Verlag, 1975.

\bibitem{BorgeJessenDistributionsAndZetaFunction}
B\o~rge Jessen and Aurel Wintner.
\newblock Distribution functions and the {R}iemann zeta function.
\newblock {\em Trans. Amer. Math. Soc.}, 38(1):48--88, 1935.

\bibitem{BorelNormalNumbers}
Emile Borel.
\newblock M. les probabilités dénombrables et leurs applications
  arithmétiques.
\newblock {\em Rend. Circ. Matem. Palermo}, pages 247--271, 1909.

\bibitem{CEFractalPaper}
Michael Coons and James Evans.
\newblock A sequential view of self--similar measures, or, what the ghosts of
  mahler and cantor can teach us about dimension.
\newblock preprint.

\end{thebibliography}

\end{document}